\documentclass{amsart}
\usepackage{amssymb,amsmath,amsthm}
\usepackage[T1]{fontenc}
\usepackage[dvips]{graphicx}
\usepackage{enumerate}

\theoremstyle{plain}
\newtheorem{theorem}{Theorem}

\newtheorem{proposition}{Proposition}
\newtheorem*{Main Lemma}{Main Lemma}

\theoremstyle{definition}

\newtheorem{rem}{Remark}

\DeclareMathOperator*{\hd}{\mathrm{dim}_\mathrm{H}}

\title[On the uniqueness of an ergodic measure of full dimension]{On the uniqueness of an ergodic measure of full dimension for non-conformal repellers}
\author{Nuno Luzia}

\begin{document}
\maketitle 
\begin{abstract}
We give a subclass $\mathcal{L}$ of \emph{Non-linear Lalley-Gatzouras carpets} and an open set $\mathcal{U}$ in $\mathcal{L}$ such that any carpet in $\mathcal{U}$  has a unique ergodic measure of full dimension. In particular, any Lalley-Gatzouras carpet which is close to a \emph{non-trivial} general Sierpinski carpet has a unique ergodic measure of full dimension.
\end{abstract}
\tableofcontents
\section{Introduction}

It is well known that a $\mathrm{C}^{1+\alpha}$ conformal repeller has a unique ergodic measure of full dimension. This is a consequence of Bowen's equation together with the classical thermodynamic formalism developed by Sinai-Ruelle-Bowen, see \cite{14}, \cite{12}, \cite{3} and \cite{13}. Moreover, this measure is a Gibbs state relative to some H\"older-continuous potential. Is this true for non-conformal repellers?

The simplest examples of non-conformal repellers are the \emph{general Sierpinski carpets}, whose Hausdorff dimension was studied by Bedford \cite{2} and McMullen \cite {10}. They computed the Hausdorff dimension of these sets by establishing the variational principle for the dimension. As a consequence, these repellers have an ergodic measure of full dimension (in fact Bernoulli) and, by \cite{11}, this measure is \emph{unique}.

In \cite{6} Lalley and Gatzouras introduced a larger class of non-conformal repellers and computed their Hausdorff dimension also by establishing the variational principle for the dimension, and so these repellers  have a Bernoulli measure of full dimension (see also \cite{9} for a \emph{random} version of this result). In \cite{1} the authors give an example of a \emph{Lalley-Gatzouras carpet} which has two Bernoulli measures of full dimension. So the answer to the question formulated above is negative.
  
In this paper, we study this problem -- \emph{existence} and \emph{uniqueness} of an ergodic measure of full dimension -- for a larger class of non-conformal repellers which we shall call \emph{Non-linear Lalley-Gatzouras carpets}. As the name suggests, these repellers are the $\mathrm{C}^{1+\alpha}$ non-linear versions of the Lalley-Gatzouras carpets. They are defined by an Iterated Function System $\{f_{ij}\}$ where $f_{ij}\colon[0,1]^2\to[0,1]^2$, $i=1,...,m,\,j=1,...,m_i$ have the \emph{skew-product} form $f_{ij}(x,y)=(a_{ij}(x,y),b_i(y))$, with the domination condition $0<|\partial_x a_{ij}(x,y)|<|b_i'(y)|<1$, and the corresponding attractor $\Lambda$ (see Section \ref{s2} for precise definitions). The Hausdorff dimension of these repellers was, essentially, computed in \cite{7} by establishing the variational principle for the dimension. Because of the non-linearity of the transformations $f_{ij}$, the \emph{existence} of an ergodic measure of full dimension turns out to be a non-trivial problem. This was proved to be true in \cite{8} (in a more general context). Then we have the following.

\begin{theorem}\label{t1}
A Non-linear Lalley-Gatzouras carpet has an ergodic measure of full dimension.
Moreover, this measure is a Gibbs state for a relativized variational principle. 
\end{theorem}

As we know now (by \cite{1}), such a measure is, in general, not unique. The main purpose of this paper is to give sufficient conditions for having a \emph{unique} ergodic measure of full dimension, based on an idea introduced in Remark 2 of \cite{8}. 

We can introduce a natural topology on the class of Non-linear Lalley-Gatzouras carpets by saying that two of these carpets are close if the corresponding functions of the Iterated Function System are $\mathrm{C}^{1+\alpha}$ close
(with alphabet $(i,j)$ fixed). We denote by $\mathcal{L}$ the subclass of Non-linear Lalley-Gatzouras carpets for which $\partial_{xx} a_{ij}=0$, i.e. $a_{ij}(x,y)=\tilde{a}_{ij}(y) x+u_{ij}(y)$. Of course, $\mathcal{L}$ contains the Lalley-Gatzouras carpets. In this paper, a \emph{general Sierpinski carpet} is a Lalley-Gatzouras carpet for which $\partial_x a_{ij}=a$ and $b_i'=b$ for some positive constants $a$ and $b$ and every $(i,j)$ (this is a more general definition than usual). We say that such a carpet is \emph{non-trivial} if $a<b$ and the natural numbers $m_i\ge 2$, $i=1,...,m$ are not all equal to each other.

\begin{theorem}\label{t2}
There is an open set $\mathcal{U}$ in $\mathcal{L}$ such that: 
\begin{enumerate}[(i)]
\item 
 $\mathcal{U}$ contains all non-trivial general Sierpinski carpets; 
\item 
 every reppeller $K$ in $\mathcal{U}$ has a unique ergodic measure of full dimension $\mu_K$; 
\item
the map  $\mathcal{U}\ni K \mapsto \mu_K$ is continuous.
\end{enumerate}
\end{theorem}

We believe that Theorem \ref{t2} also holds in the class of Non-linear Lalley-Gatzouras carpets. The reason for restricting to the subclass $\mathcal{L}$ relies on the necessity of considering \emph{basic potentials} in the relativized variational principle of \cite{5}, which we use, in order to have additional properties (see Remark \ref{rem2}).
    
This paper is organized as follows. In Section \ref{s2} we introduce the class of Non-linear Lalley-Gatzouras carpets and say how Theorem \ref{t1} follows from the works \cite{7} and \cite{8}. In Section \ref{s3}, within the more general context of \cite{8}, we prove some properties of \emph{measures of maximal dimension}, a relativized version of Ruelle's formulas for the derivative of the pressure, and a criterium for uniqueness of a measure of maximal dimension (Theorem \ref{t5}). In Section \ref{s4} we use this criterium to prove Theorem \ref{t2}.

\section{Non-linear Lalley-Gatzouras carpets}\label{s2}

\subsection{Definition}
Let $g_i\colon [0,1]\to[0,1]$, $i=1,...,m$ be $\mathrm{C}^{1+\alpha}$ for some  $\alpha>0$. We say that
$\{g_1,...,g_m\}$ is a \emph{Simple Function System} (SFS) if:
\begin{itemize}
\item $0<|g_i'(x)|<1$ for every $x\in[0,1]$;
\item the sets $g_i([0,1])$, $i=1,...,m$ are pairwise disjoint.
\end{itemize} 

Let $f_{ij}\colon[0,1]^2\to[0,1]^2$, $i=1,...,m$, $j=1,...,m_i$ be $\mathrm{C}^{1+\alpha}$ such that:
\begin{itemize}
\item[(H1)] $f_{ij}(x,y)=(a_{ij}(x,y), b_i(y))$;
\item[(H2)] $\{b_1,...,b_m\}$ is SFS;
\end{itemize}
for each $i\in\{1,...,m\}$ and $y\in[0,1]$,
\begin{itemize}
\item[(H3)]  $\{a_{i1}(\cdot,y),...,a_{im_i}(\cdot,y)\}$ is SFS;
\item[(H4)] $\max_{x\in[0,1]}|\partial_x a_{ij}(x,y)|<|b_i'(y)|$, $j=1,...,m_i$. 
\end{itemize} 
Then there is a unique nonempty compact set $\Lambda$ of $[0,1]^2$ such that
\[
   \Lambda=\bigcup_{(i,j)} f_{ij}(\Lambda).
\]
We call the pair $(\{f_{ij}\}, \Lambda)$ a \emph{Non-linear Lalley-Gatzouras carpet}. 

When $\partial_{xx} a_{ij}=0$ we get the definition of a carpet in $\mathcal{L}$. When the functions $a_{ij}$ and $b_i$ are linear and $\partial_y a_{ij}=0$, we get the definition of a \emph{Lalley-Gatzouras carpet}, see \cite{6} (where equality is allowed in (H4)). 
When, moreover, $\partial_x a_{ij}=a$ and $b_i'=b$ for some positive constants $a$ and $b$ and every $(i,j)$, we get the definition of a \emph{general Sierpinski carpet}, see \cite{2} and \cite{10} (in fact, our definition is a little more general).

\subsection{Hausdorff dimension}
The Hausdorff dimension of Non-linear Lalley-Gat\-zou\-ras carpets  was, essentially, computed in \cite{7} by establishing the variational principle for the dimension. In fact, the theorems in \cite{7} are formulated in terms of a Dynamical System $f$ instead of an Iterated Function System $\{f_{ij}\}$, although in its proofs we mainly used the ${f_{ij}}$ approach. The relation between the two approaches is given by $f_{ij}=(f|R_{ij})^{-1}$ where $R_{ij}$ is an element of a Markov partition for $f$. Beside imposing a skew-product structure for $f$ (which translates to (H1)), we considered a $\mathrm{C}^{2}$ perturbation of the 2-torus transformation $f_0(x,y)=(lx,my)$, where $l>m>1$ are integers.  The only reason for doing this is to inherit from the linear system a \emph{domination condition} (which translates to (H4)) and a simple Markov partition (inducing a full shift) which is \emph{smooth}. More precisely, the Markov partition is constructed using the invariant foliation by horizontal lines (due to the skew-product structure) and an invariant \emph{smooth vertical} foliation, which exists because the vertical lines constitute a \emph{normally expanding} invariant foliation for $f_0$. In the present setting, all we need to show is that the sets
\[
          R_{(i_1j_1)(i_2j_2)...(i_nj_n)}=f_{i_1j_1}\circ f_{i_2j_2}\circ\cdots \circ f_{i_nj_n} ([0,1]^2)
\]
have \emph{vertical boundaries} formed by $\mathrm{C}^1$ curves with uniformly bounded \emph{distortion} for all $n\in\mathbb{N}$. But, as we shall see, this is a consequence of the domination condition (H4). 

Let
\[
     \lambda= \max_{(x,y),(i,j)} \frac{|\partial_x a_{ij}(x,y)|}{|b_i'(y)|}
\]     
which is $<1$ by (H4), and
\[
       C=(1-\lambda)^{-1} \max_{(x,y),(i,j)} \frac{|\partial_y a_{ij}(x,y)|}{|b_i'(y)|} .
\]
We will see that each $f_{ij}$ transforms \emph{vertical graphs} with \emph{distortion} $\le C$ into vertical graphs with distortion $\le C$. Let $\mathcal{G}_F=\{(F(y),y)\colon y\in I\}$  with $|F'(y)|\le C$ for all $y\in I$ (where $I$ is some subinterval of $[0,1]$). Then $f_{ij}(\mathcal{G}_F)=\mathcal{G}_G$ where 
\[
         G(y)=a_{ij}(F(b_i^{-1}(y)), b_i^{-1}(y)),\quad y\in b_i(I).
\]
We see that (with $z=b_i^{-1}(y)$ and $w=(F(z), z)$) 
\[
    G'(y)=\partial_x a_{ij}(w) b_i'(z)^{-1} F'(z)+\partial_y a_{ij}(w) b_i'(z)^{-1},
\]
so $|G'(y)|\le \lambda C +|\partial_y a_{ij}(w)b_i'(z)^{-1}|\le C$. Then, starting with the vertical graphs $\{0\}\times[0,1]$ and $\{1\}\times[0,1]$ and using induction on $n$, we get the desired property for the sets $R_{(i_1j_1)(i_2j_2)...(i_nj_n)}$.

Then it follows from the proof of Theorem A in \cite{7} that, there exists $A>0$ such that, for every $n\in\mathbb{N}$,
\begin{equation*}
     \hd\Lambda=  \hd\Lambda_n \pm\frac{A}{n},
\end{equation*}
where $\Lambda_n$ is a Lalley-Gatzouras carpet defined using an appropriate linearization of the functions
$f_{i_1j_1}\circ f_{i_2j_2}\circ\cdots \circ f_{i_nj_n}$. 

More precisely, given $n\in\mathbb{N}$, consider the $n$-tuples ${\bf{i}}=(i_1,...,i_n)$ and ${\bf{j}}=(j_1,...,j_n)$, where $i_k\in\{1,..,m\}$, $j_k\in\{1,...,m_{i_k}\}$, $k=1,...,n$, and write
\[
   b_{\bf{i}}=b_{i_1}\circ b_{i_2}\circ\cdots \circ b_{i_n},\quad a_{{\bf{i}}{\bf{j}}}=\pi_1(f_{i_1j_1}\circ f_{i_2j_2}\circ\cdots \circ f_{i_nj_n}),
\]
where $\pi_1(x,y)=x$. Note that, because of the skew-product structure,
\[
   b_{\bf{i}}'(y)=\prod_{k=1}^n b'_{i_k}(y_k)\quad\text{and}\quad \partial_x a_{{\bf{i}}{\bf{j}}}(x,y)=\prod_{k=1}^n \partial_x a_{i_kj_k}(z_k),
\]  
where  $y_k=b_{i_{k+1}}\circ\cdots\circ b_{i_n}(y)$, $y_n=y$ and $z_k= f_{i_{k+1}j_{k+1}}\circ\cdots\circ f_{i_nj_n}(x,y)$, $z_n=(x,y)$. Consider the numbers
\[
  \alpha_{{\bf{i}\bf{j}},n}=\max_{(x,y)\in [0,1]^2}\, |\partial_x a_{{\bf{i}}{\bf{j}}}(x,y)|\quad\text{and}\quad
  \beta_{{\bf{i}},n}=\max_{y\in[0,1]}\, |b_{\bf{i}}'(y)|.
\]
Let ${\bf{p}}^n=(p_{\bf{i}}^n)$ be a probability vector in $\mathbb{R}^{nm}$. We define
\[
    \lambda_n({\bf{p}}^n)=\frac{\sum_{\bf{i}} p_{\bf{i}}^n \log p_{\bf{i}}^n}
     {\sum_{\bf{i}} p_{\bf{i}}^n \log \beta_{{\bf{i}},n}},
\] and $t_n({\bf{p}}^n)$ as being the unique real in $[0,1]$ satisfying
\begin{equation*}
     \sum_{\bf{i}} p_{\bf{i}}^n \log \left(\sum_{\bf{j}}
      \alpha_{{\bf{i}\bf{j}},n}^{t_n({\bf{p}}^n)} \right)=0 .
\end{equation*}
Consider the Bernoulli measure $\mu_{{\bf{p}}^n}$ for the Iterated Function System $\{f_{i_1j_1}\circ\cdots \circ f_{i_nj_n}\}$ that assigns to each $R_{(i_1j_1)...(i_nj_n)}$ the weigth
\[
    p_{\bf{i}}^n\,\frac{\alpha_{{\bf{i}\bf{j}},n}^{t_n({\bf{p}}^n)}}
   {\sum_{\bf{j'}} \alpha_{{\bf{i}\bf{j'}},n}^{t_n({\bf{p}}^n)}}.
\]
\begin{theorem}[Proof of Theorem A, \cite{7}]
Let $(\{f_{ij}\},\Lambda)$ be a Non-Linear Lalley-Gatzouras carpet. There exist constants $A,B>0$ such that, for every $n\in\mathbb{N}$,
\[
   \hd \mu_{{\bf{p}}^n}=\lambda_n({\bf{p}}^n) + t_n({\bf{p}}^n) \pm\frac{B}{n},
\]
and
\begin{equation*}
     \hd\Lambda= \sup_{{\bf{p}}^n} \left\{\lambda_n({\bf{p}}^n) + t_n({\bf{p}}^n) \right\} \pm\frac{A}{n}.
\end{equation*}
Moreover, $(\{f_{ij}\},\Lambda)\mapsto \hd\Lambda$ is a continuous function in the class of Non-Linear Lalley-Gatzouras carpets.
\end{theorem}
\begin{rem}
The continuity of $(\{f_{ij}\},\Lambda)\mapsto \hd\Lambda$ follows from the Proof of Corollary A in \cite{7}. In fact, there we used the $\mathrm{C^2}$ topology but it is clear that we can use the $\mathrm{C^{1+\alpha}}$ topology.
\end{rem}
As a consequence, the \emph{variational principle for dimension} holds, i.e. the Hausdorff dimension of $\Lambda$ is the 
supremum of the Hausdorff dimension of ergodic measures (with respect to $\{f_{ij}\}$) on $\Lambda$. In \cite{8} we prove the \emph{existence} of an ergodic measure of full dimension for $\Lambda$, which is a Gibbs state for a relativized variational principle. Thus we have Theorem \ref{t1}. 

\section{Properties of measures of maximal dimension}\label{s3}
The results given in this section hold in the more general context of \cite{8}. We consider $(X,T)$ and $(Y,S)$ mixing subshifts of finite type such that $(Y,S)$ is a factor of $(X,T)$ with factor map $\pi\colon X\to Y$. Assume that each fibre $\pi^{-1}(y)$ has at least two points.

\subsection{Characterization of measures of maximal dimension}
 We use the following notation: $\mathcal{M}(T)$ is the set of all $T$-invariant Borel
probability measures on $X$; $h_\mu(T)$ is the metric entropy of $T$ with respect to $\mu\in\mathcal{M}(T)$.

Let $\varphi\colon X\to \mathbb{R}$ and $\psi\colon Y\to \mathbb{R}$ be positive H\"older-continuous functions. We define
\begin{align*}
  D(\mu)&=\frac{h_{\mu\circ\pi^{-1}}(S)}{\int \psi\circ\pi\,d\mu}+\frac{h_{\mu}(T)-h_{\mu\circ\pi^{-1}}(S)}{\int \varphi\,d\mu},
\end{align*}
and
\[
   D=\sup_{\mu\in\mathcal{M}(T)} D(\mu).
\]
Note that if $\mu$ is ergodic then $D(\mu)$ might be interpretated as the Hausdorff dimension of the measure $\mu$ (see Remark 5 of \cite{8}). We say that $\mu$ is a \emph{measure of maximal dimension} if $D(\mu)=D$. In \cite{8} we prove the existence of an ergodic measure of maximal dimension, and give a characterization of measures of maximal dimension that we shall describe now (for more details see this reference).

We use the following version of the \emph{relativized variational principle} by \cite{4} and \cite{5}. Given an Hölder-continuous function $\phi\colon X\to \mathbb{R}$ and $\nu\in\mathcal{M}(S)$, there exists a positive Hölder-continuous function $A_{\phi}\colon Y\to\mathbb{R}$ (not depending on $\nu$) such that
\begin{equation}\label{rpv}
  \sup_{\substack{\mu\in\mathcal{M}(T)\\ \mu\circ\pi^{-1}=\nu}}
  \left\{h_\mu(T)-h_\nu(S)+\int_X \phi\,d\mu\right\}=\int_{Y} \log A_{\phi}\,d\nu.
\end{equation}
Moreover, there is a unique measure $\mu$ for which the supremum in (\ref{rpv}) is attained which we call the \emph{relative equilibrium state} with respect to $\phi$ and $\nu$, and $\mu$ is ergodic if $\nu$ is ergodic.

Given $\nu\in\mathcal{M}(S)$, there is a unique real $t(\nu)\ge0$ such that
\begin{equation*}\label{pres0}
    \int_{Y} \log A_{-t(\nu)\varphi}\,d\nu=0.
\end{equation*}
Then it easy to see that
\begin{equation}\label{pvorig2}
   D=\sup_{\nu\in\mathcal{M}(S)} \left\{
   \frac{h_{\nu}(S)}{\int \psi \,d\nu}+ t(\nu) \right\}.
\end{equation}
Let
\[
  \underline{t}=\inf_{\nu\in\mathcal{M}(S)} t(\nu) \quad\text{and}\quad
  \overline{t}=\sup_{\nu\in\mathcal{M}(S)} t(\nu).
\]
Throughout this paper we assume $D$ and $\overline{t}$ are uniformly bounded (with respect to $\psi$ and $\varphi$), since in applications these numbers have dimension interpretations.
We assume the following technical condition:
\begin{equation}\label{hip}
 \text{the supremum in (\ref{pvorig2}) is not attained at an ergodic measure $\nu$ with $t(\nu)=\underline{t}$ or $\overline{t}$}.\tag{H}
\end{equation}

Let $P(\cdot)$ denote the classical Pressure function with respect to $(Y,S)$, and let $\nu_g$ denote the corresponding Gibbs state with respect to the Hölder-continuous potential $g\colon Y\to \mathbb{R}$.
Given $t\in(\underline{t},\overline{t})$, let
\begin{equation} \label{potential}
  \Phi_t = (t-D)\psi + \beta(t) \log A_{-t\varphi} 
\end{equation}
where $\beta(t)$ is the unique real satisfying
\[
   \int \log A_{-t\varphi} \, d \nu_{\Phi_t} =0
\]
(see \cite{8} for details).
Finally, let $\mu_{\Phi_t}$ be the relative equilibrium state with respect to  $-t\varphi$ and $\nu_{\Phi_t}$.
The following result follows from the proof of Theorem A and Remark 3 in \cite{8}.
\begin{theorem}[Proof of Theorem A, \cite{8}]\label{tl2}
Assume (\ref{hip}). Then $D(\mu)=D$ if and only if $\mu=\mu_{\Phi_t}$ and $P(\Phi_t)=0$ (the maximum value).
\end{theorem}

\subsection{Relativized Ruelle's formulas}

We begin by recalling some classical Ruelle's formulas for the derivative of the pressure. 

Let $Z=X$ or $Y$. Given $C>0$ and $0<\theta\le 1$, let $\mathcal{H}^{C,\theta}(Z)$ denote the space of H\"older-continuous functions $\phi\colon Z\to\mathbb{R}$ satisfying
\begin{equation}\label{hold}
   |\phi(z_1)-\phi(z_2)|\le C d(z_1,z_2)^\theta,\text{ for all } z_1,z_2\in Z,
\end{equation}
and let 
\[
   ||\phi||_\theta=\inf\{ C>0\colon (\ref{hold}) \text{ holds}\}.
\]
$\mathcal{H}^{C,\theta}(Z)$ becomes a Banach space with the norm $\|| \phi |\|_\theta=\max(||\phi||, ||\phi||_\theta)$, where $||.||$ is the uniform norm.

Let $\phi_t\colon Z \to \mathbb{R}$ be a one-parameter family of continuous functions. We say that $t\mapsto\phi_t$ is  \emph{differentiable} if its partial derivative in $t$ exists, let us call it  $\dot{\phi}_t$ or $\frac{d}{dt} \phi_t$, and it is a one-parameter family of continuous functions. 

Then the following result follows from \cite{13}. 
\begin{proposition}[Ruelle \cite{13}]\label{prop1}
If $\phi_t\in \mathcal{H}^{C,\theta}(Y)$ (with $C, \theta$ independent of $t$), $t\mapsto\phi_t$ is differentiable and 
$\dot{\phi}_t\in \mathcal{H}^{C,\theta}(Y)$  then
\[
\frac{dP(\phi_t) }{dt} = \int \dot{\phi}_t\, d \nu_{\phi_t}.
\]
If, moreover,  $h\in \mathcal{H}^{C,\theta}(Y)$ then
\[
      \frac{d}{dt} \int h\, d \nu_{\phi_t}=Q_{\phi_t}(\dot{\phi}_t,h),
\]
where $Q_{\phi_t}(\cdot,\cdot)\colon \mathcal{H}^{C,\theta}(Y)\times\mathcal{H}^{C,\theta}(Y)\to\mathbb{R}$ is given by
\[
     Q_{\phi_t}(h_1,h_2)=\sum_{n=0}^\infty \left(\int_Y (h_1\circ S^n)\,h_2\,d \nu_{\phi_t}-
     \int_Y h_1\,d \nu_{\phi_t} \,\int_Y h_2\,d \nu_{\phi_t}\right).
\]
There exists a constant $B>0$ (depending only on $C$ and $\theta$) such that
\[
            |Q_{\phi_t}(h_1,h_2)|\le B ||h_1||_\theta \,||h_2||_\theta.
\]
Also, for each $h_1,h_2\in\mathcal{H}^{C,\theta}(Y)$, 
\[
        \mathcal{H}^{C,\theta}(Y)\ni \phi\mapsto Q_{\phi}(h_1,h_2)
\] 
is a continuous function.
\end{proposition}

Now we recall some definitions from \cite{4} and \cite{5} that are used to define $A_\phi$, for $\phi\in \mathcal{H}^{C,\theta}(X)$. Given $y\in Y$, let $C_y$ denote the space of bounded continuous functions $f\colon \pi^{-1}(y)\to\mathbb{R}$. For each $y\in Y$ and $n\in\mathbb{N}$, define the operators $G_y^{(n)}$ and $P_y^{(n)}\colon C_y\to C_y$ by
\begin{equation}\label{G}
      (G_y^{(n)} f)(x):=\sum_{\substack{T^n(x')=T^n(x)\\ \pi(x')=y}} \exp\left(\sum_{k=0}^{n-1} \phi(T^k(x'))\right) f(x'),
\end{equation}
and
\[
        (P_y^{(n)} f)(x):=\frac{(G_y^{(n)} f)(x)}{(G_y^{(n)} {\bf{1}})(x)}.
\]
Then (see Proposition 2.5 of \cite{5}), 
\[
      A_\phi(y):=\lim_{n\to\infty} \frac{(G_y^{(n+1)} {\bf{1}})(x)}{(G_{S(y)}^{(n)} {\bf{1}})(T(x))},
\]
uniformly in $y\in Y$, $x\in \pi^{-1}(y)$. Moreover (see Corollary 4.14, Remark 4.16 and Proposition 5.5 of \cite{4}), the rate of convergence is exponential depending only in $C$ and $\theta$. Also, for any $y\in Y$, the operators $P_y^{(n)}$ converge to a conditional expectation operator $P_y$ which gives a probability measure $\mu_y$ in $\pi^{-1}(y)$, in the sense that
\[
            (P_y f)(x)=\int f\,d\mu_y,\quad\text{for any } x\in\pi^{-1}(y).
\]
The system $\{\mu_y\colon y\in Y\}$ is called a \emph{Gibbs family} for $\phi$.

We will use the following property of $A_\phi$. 
Given $y\in Y$, consider the operators $V_y\colon C_y\to C_{S(y)}$ and $U_y\colon C_{S(y)}\to C_y$ given by
\[
       (V_y f)(x):=\sum_{\substack{x'\in T^{-1}(x)\\ \pi(x')=y}} \exp\left(\phi(x')\right) f(x'),
\]
and
\[
  (U_y f)(x):=f(T(x)).
\]
Then (see Proposition 5.5 of \cite{4})
\begin{equation}\label{A1}
   A_\phi(y) P_y=U_y P_{S(y)} V_y.
\end{equation}
(Note that the operators $G_y^{(n)}$, $P_y^{(n)}$, $P_y$ and $V_y$ depend on the potential $\phi$.)

We say that $\phi\in \mathcal{H}^{C,\theta}(X)$ is a \emph{basic potential} (see Definition 4.1 of \cite{5}), if for $y\in Y$ and
$x\in \pi^{-1}(S(y))$ we have
\begin{equation}\label{basic}
 A_\phi(y)=(V_y {\bf{1}})(x),
\end{equation}
i.e., for each $y\in Y$, the function $V_y {\bf{1}}$ is constant.
In this case we have the following.
\begin{proposition}[\cite{5}]\label{prop2}
If $\phi\in \mathcal{H}^{C,\theta}(X)$ is a basic potential then:
\begin{enumerate}[(a)]
\item
the Gibbs family for $\phi$ is \emph{covariant}, i.e.
\[
     \mu_y\circ T^{-1}=\mu_{S(y)}
\]
for each $y\in Y$;
\item
the relative equilibrium state for (\ref{rpv}) with respect to  $\phi$ and $\nu$ is given by $\mu=\mu_y\times\nu$;

\end{enumerate}
\end{proposition}

Now we are ready to prove the following.
\begin{proposition}\label{prop3}
Let $\varphi\in \mathcal{H}^{C,\theta}(X)$ and assume $-t\varphi$ is a basic potential for $t\in (\underline{t},\overline{t})$.
Then $t\mapsto A_{-t\varphi}$ is differentiable and
\begin{equation}\label{RR1}
   \frac{d}{dt}\log A_{-t\varphi}=-\int \varphi\,d\mu_{t,y},
\end{equation}
where $\{\mu_{t,y}\}$ is the Gibbs family for $-t\varphi$. Moreover,  $\frac{d}{dt}\log A_{-t\varphi}\in \mathcal{H}^{D_\theta C,\eta(\theta)}(Y)$, for some $D_\theta>0$ and $\eta(\theta)\in(0,1]$, $t\mapsto \frac{d}{dt}\log A_{-t\varphi}$ is differentiable and
\begin{equation}\label{RR2}
    \frac{d^2}{dt^2} \log A_{-t\varphi}=\int \varphi^2\,d\mu_{t,y}  -\left( \int \varphi\,d\mu_{t,y}\right)^2.
\end{equation}
\end{proposition}
\begin{proof}
The differentiability of $t\mapsto A_{-t\varphi}$ is an immediate consequence of (\ref{basic}), and
\begin{equation}\label{difA}
       \frac{d}{dt} A_{-t\varphi}(y)=-(V_{t,y} \varphi)(x),\quad y\in Y,\,x\in\pi^{-1}(S(y))
\end{equation}
(where $V_{t,y}$ is $V_{y}$ with the potential $\phi=-t\varphi$). In particulary, $(V_{t,y} \varphi)(x)$ does not depend on 
$x\in\pi^{-1}(S(y))$. Then applying (\ref{A1}) to $\varphi$ we get
\[
    A_{-t\varphi}(y) P_{t,y}\varphi= V_{t,y} \varphi,
\]
which together with (\ref{difA}) gives (\ref{RR1}). The H\"older-continuity of $\frac{d}{dt}\log A_{-t\varphi}$ follows from Theorem 2.10 of \cite{4}.

In the same way, by (\ref{difA}) we see that $t\mapsto\frac{d}{dt}\log A_{-t\varphi}$ is differentiable and
\[
   \frac{d^2}{dt^2} A_{-t\varphi}(y)=(V_{t,y} \varphi^2)(x),\quad y\in Y,\,x\in\pi^{-1}(S(y)),
\]
and, by (\ref{A1}) applied to $\varphi^2$, 
\[
       A_{-t\varphi}(y) P_{t,y}\varphi^2= V_{t,y} \varphi^2,
\]
so that
\[
\frac{d^2}{dt^2} A_{-t\varphi}(y)=A_{-t\varphi}(y) P_{t,y}\varphi^2.
\]
Since
\[
\frac{d^2}{dt^2} \log A_{-t\varphi}(y)=\frac{1}{A_{-t\varphi}(y)}\frac{d^2}{dt^2} A_{-t\varphi}(y)-\left(\frac{d}{dt}\log A_{-t\varphi}(y)\right)^2,
\]
we get (\ref{RR2}).
\end{proof}

Recall the definition of $\Phi_t$ from (\ref{potential}). 
\begin{proposition}\label{lem2}
Assume $-t\varphi$ is a basic potential for $t\in (\underline{t},\overline{t})$. Then $t\mapsto \Phi_t$ is differentiable and
\begin{equation}\label{RR3}
  \frac{dP(\Phi_t) }{dt} = \int \psi\, d \nu_{\Phi_t} - \beta(t) \int \varphi\, d \mu_{\Phi_t}.
\end{equation}
Moreover, 
\begin{multline}\label{RR4}
     \frac{d^2P(\Phi_t) }{dt^2}=-\beta'(t) \int \varphi\, d \mu_{\Phi_t}\,+\,\beta(t)\int \frac{d^2}{dt^2} \log A_{-t\varphi}\, d \nu_{\Phi_t}\\ +\,Q_{\Phi_t}(\psi, \dot{\Phi_t})\,+\, \beta(t) \,Q_{\Phi_t}\left(\frac{d}{dt} \log A_{-t\varphi}, \dot{\Phi_t}\right).
\end{multline}
\end{proposition}
\begin{proof}
Let $\psi, \varphi \in\mathcal{H}^{C,\theta}(Z)$, where $Z=Y$ or $X$. Fix $\varepsilon>0$ arbitrarly small.
It follows from Theorem 2.10 of \cite{4} (see also Proposition 2 of \cite{8}) that $A_{t\varphi} \in\mathcal{H}^{D_1, \eta}(Y)$, for some constants $D_1=D_1(C,\theta)>0$ and $\eta=\eta(\theta)>0$, for every $t\in [\underline{t},\overline{t}]$. Of course, we may assume $\eta\le\theta$. It is also proved in \cite{8} that $\beta(t)$ is continuous for $t\in [\underline{t}+\varepsilon,\overline{t}-\varepsilon]$. So, by (\ref{potential}), we have $\Phi_{t} \in\mathcal{H}^{D_2,\eta}(Y)$, for some constant $D_2=D_2(C,\theta,\varepsilon)$, for every $t\in [\underline{t}+\varepsilon,\overline{t}-\varepsilon]$. 

Let us see that $\beta(t)$ is $C^1$ for $t\in(\underline{t},\overline{t})$. 
Let
\[
   F(t,\beta)=\int \log A_{-t\varphi}\,d\nu_{(t,\beta)},
\]
where $\nu_{(t,\beta)}$ is the Gibbs sate for the potential
$\phi_{(t,\beta)}=(t-D)\psi+\beta \log A_{-t\varphi}$. By Propositions \ref{prop1} and \ref{prop3},
\begin{equation}\label{imp1}
    \frac{\partial F}{\partial \beta}(t,\beta)=Q_{\phi_{(t,\beta)}}
    (\log A_{-t\varphi}, \log A_{-t\varphi})
\end{equation}
and
\begin{equation}\label{imp2}
    \frac{\partial F}{\partial t}(t,\beta)=\int\frac{d}{dt} \log A_{-t\varphi}\,d\nu_{(t,\beta)}+
    Q_{\phi_{(t,\beta)}}\left(\log A_{-t\varphi}, \psi +\beta \frac{d}{dt} \log A_{-t\varphi}\right),
\end{equation}
and so, by \cite{13}, $F$ is $C^1$. By \cite{8}, $\frac{\partial F}{\partial \beta}(t,\beta)>0$ and $\beta(t)$ is well defined as the unique solution of $F(t,\beta(t))=0$. Then, it follows by the implicit function theorem that $\beta(t)$ is $C^1$ and 
\begin{equation}\label{imp3}
\beta'(t)=-\frac{\partial F}{\partial t}(t,\beta(t))\, / \,\frac{\partial F}{\partial \beta}(t,\beta(t)).
\end{equation}

Then $t\mapsto \Phi_{t}$ is differentiable, 
\begin{equation}\label{phidot}
  \dot{\Phi_t}=\psi+\beta'(t)  \log A_{-t\varphi} + \beta(t) \frac{d}{dt} \log A_{-t\varphi}\in\mathcal{H}^{D_2,\eta}(Y)
\end{equation}
for every $t\in [\underline{t}+\varepsilon,\overline{t}-\varepsilon]$ (after, eventually, increasing $D_2$ and decreasing $\eta$), and applying Proposition \ref{prop1} we get
\begin{equation}\label{above}
  \frac{dP(\Phi_t) }{dt} = \int \psi\, d \nu_{\Phi_t} + \beta(t) \int \frac{d}{dt} \log A_{-t\varphi}\, d \nu_{\Phi_t}.
\end{equation}
This together with (\ref{RR1}) and Proposition \ref{prop2} gives (\ref{RR3}). In the same way, (\ref{RR4}) follows by applying 
Proposition \ref{prop1} to (\ref{above}).
\end{proof}

Let
\[
    \rho(t)=\frac{\int \psi\, d \nu_{\Phi_t}}{\int \varphi\, d \mu_{\Phi_t}}.
\]
\begin{proposition}\label{prop}
 Assume (\ref{hip}) and $-t\varphi$ is a basic potential for $t\in (\underline{t},\overline{t})$. If $D(\mu)=D$ then $\mu=\mu_{\Phi_t}$ and $\beta(t)=\rho(t)$.
\end{proposition}
\begin{proof}
From Theorem \ref{tl2} we have that $\mu=\mu_{\Phi_t}$ and $\frac{dP(\Phi_t) }{dt}=0$. Then it follows from Proposition \ref{lem2} that $\beta(t)=\rho(t)$.
\end{proof}

\subsection{Criterium for uniqueness of measure of maximal dimension}
Now we give sufficient conditions for having $\frac{d^2P(\Phi_t) }{dt^2}<0$ which, by Theorem \ref{tl2}, implies the existence of a \emph{unique} measure of maximal dimension (the existence follows from Theorem A in \cite{8}), as already noticed in Remark 3 in \cite{8}.

We will need a \emph{uniform version} of Hypothesis (\ref{hip}). Given $\varepsilon>0$ let
\begin{equation}\label{hipu}
  \text{if the supremum in (\ref{pvorig2}) is attained at an ergodic measure $\nu$ then $t(\nu)\in(\underline{t}+\varepsilon,\overline{t}-\varepsilon)$}. \tag{$\mathrm{H}_\varepsilon$} 
\end{equation}

\begin{theorem}\label{t5}
Let $\psi\in \mathcal{H}^{C,\theta}(Y)$ and $\varphi\in \mathcal{H}^{C,\theta}(X)$ be positive. Assume $-t\varphi$ is a basic potential for $t\in(\underline{t},\overline{t})$.
Assume  (\ref{hipu}) for some $\varepsilon>0$. Take any $\gamma>0$ such that 
\begin{enumerate}[(i)]
       \item $\gamma^{-1}<\varphi<\gamma$;
       \item $|\beta(t)|<\gamma,\quad t\in(\underline{t}+\varepsilon,\overline{t}-\varepsilon)$;
       \item $|\beta'(t)|<\gamma,\quad t\in(\underline{t}+\varepsilon,\overline{t}-\varepsilon)$.
\end{enumerate}
Then there exists a constant $\delta=\delta(C,\theta, \varepsilon, \gamma)>0$ such that, if $||\psi||_\theta<\delta$ and
$||\varphi||_\theta<\delta$, then there is a unique measure of maximal dimension, say 
$\mu_{\psi,\varphi}$, which is ergodic (a Gibbs state for a relativized variational principle). 

Moreover, hypotheses (\ref{hipu}) and (i)-(iii) are \emph{robust} in the following sense:
if $\tilde\psi\in \mathcal{H}^{C,\theta}(Y)$, $\tilde\varphi\in \mathcal{H}^{C,\theta}(X)$ are positive, $-t\tilde\varphi$ is a basic potential and $\tilde\psi,\,\tilde\varphi$ are $\||.|\|_\theta$-close to, respectively, $\psi,\,\varphi$ satisfying these hypotheses, then $\tilde\psi,\,\tilde\varphi$ also satisfy these hypotheses. Then we have that $(\psi,\varphi)\mapsto \mu_{\psi,\varphi}$ is continuous.
\end{theorem}

\begin{proof}
\emph{Unicity of $\mu_{\psi,\varphi}$.}

It follows from the Proof of Theorem A in \cite{8} (see also Remark 3 in \cite{8}) that maximizing measures are of the form
$\mu_{\Phi_t}$ where $P(\Phi_t)=\frac{dP(\Phi_t) }{dt}=0$ for some $t\in(\underline{t}+\varepsilon,\overline{t}-\varepsilon)$.
Therefore, we only need to prove that $\frac{d^2 P(\Phi_t) }{dt^2}<0$ for $t\in(\underline{t}+\varepsilon,\overline{t}-\varepsilon)$, and this will be done estimating the 4 terms in (\ref{RR4}).

\emph{Term 1.} By (\ref{imp1}), (\ref{imp2}), (\ref{imp3}), (\ref{RR1}) and Proposition \ref{prop2} (b), we have
\begin{equation}\label{estbeta}
   \beta'(t)=\frac{ \int\varphi\,d\mu_{\Phi_t} - Q_{\Phi_t}\left(\log A_{-t\varphi}, \psi +\beta \frac{d}{dt} \log A_{-t\varphi}\right) }
    { Q_{\Phi_t} (\log A_{-t\varphi}, \log A_{-t\varphi}) }.
\end{equation}
Of course, $\int\varphi\,d\mu_{\Phi_t}>\gamma^{-1}$. It follows from \cite{8} that $\log A_{-t\varphi}$ is not 
\emph{cohomologous to a constant} and so $Q_{\Phi_t} (\log A_{-t\varphi}, \log A_{-t\varphi}) >0$. 
It follows from Theorem 2.10 of \cite{4} (see also Proposition 2 of \cite{8}) that $\log A_{t\varphi} \in\mathcal{H}^{D_1 
||\varphi||_\theta, \eta}(Y)$, for some constants $D_1=D_1(\theta,\varepsilon,\gamma)>0$ and $\eta=\eta(\theta)>0$
(we put the dependence on $\varepsilon$ because $\gamma$ depends on $\varepsilon$). Of course, we may assume $\eta\le\theta$. In the same way, by (\ref{potential}), we see that $\Phi_t\in\mathcal{H}^{D_2, \eta}(Y)$,
for some constant $D_2=D_2(C,\theta,\varepsilon,\gamma)>0$.
So we may apply Proposition \ref{prop1} to obtain 
\[
Q_{\Phi_t} (\log A_{-t\varphi}, \log A_{-t\varphi})\le D_2 ||\varphi||_\theta^2
\] 
(after, eventually, increasing $D_2$; we will do this a finite number of times). In the same way, see Proposition \ref{prop3}, we have
\begin{equation}\label{dta}
   \frac{d}{dt} \log A_{-t\varphi}\in\mathcal{H}^{D_1 ||\varphi||_\theta, \eta}(Y),
\end{equation}
and, applying Proposition \ref{prop1} again, we get
\[
        \left|Q_{\Phi_t}\left(\log A_{-t\varphi}, \psi +\beta \frac{d}{dt} \log A_{-t\varphi}\right)\right|\le
        D_2 ||\varphi||_\theta.
\]
Putting all these together in (\ref{estbeta}), we get
\begin{equation}\label{term1}
  -\beta'(t) \int\varphi\,d\mu_{\Phi_t} \le - \frac{\gamma^{-2}-D_2 ||\varphi||_\theta}{D_2 ||\varphi||_\theta^2},
\end{equation}
if $||\varphi||_\theta<\gamma^{-2} D_2^{-1}$.

\emph{Term 2.} 
Remember from (\ref{RR2}),
\[
 \frac{d^2}{dt^2} \log A_{-t\varphi}=\int \varphi^2\,d\mu_{t,y}  -\left( \int \varphi\,d\mu_{t,y}\right)^2\ge0,
\]
by Cauchy-Shwarz inequality. Clearly,
\begin{align*}
   \int \varphi^2\,d\mu_{t,y}  -\left( \int \varphi\,d\mu_{t,y}\right)^2\le 
   \left(\sup \varphi\right)^2 -  \left(\inf \varphi\right)^2\le 2\gamma \left(\sup \varphi -  \inf \varphi\right),
\end{align*}
and $\sup \varphi -  \inf \varphi\le \max\{1,\mathrm{diam}(X)\} ||\varphi||_\theta$.
So,
\[
\left| \beta(t)\int \frac{d^2}{dt^2} \log A_{-t\varphi}\, d \nu_{\Phi_t} \right|\le C_0\gamma^2 ||\varphi||_\theta,
\]
for some constant $C_0$.

\emph{Term 3.} It follows from (\ref{phidot}) and reasoning as in \emph{Term 1} that 
$\dot{\Phi_t}\in\mathcal{H}^{D_2, \eta}(Y)$. So, applying Proposition \ref{prop1} we get
\[
\left| Q_{\Phi_t}(\psi, \dot{\Phi_t}) \right|\le D_2 ||\psi ||_\theta.
\]

\emph{Term 4.} It follows from (\ref{dta}) and Proposition \ref{prop1} that
\[
    \left| \beta(t)\, Q_{\Phi_t}\left(\frac{d}{dt} \log A_{-t\varphi}, \dot{\Phi_t} \right)\right|\le \gamma D_2 ||\varphi ||_\theta.
\]

Finally, putting all 4 terms together gives
\[
  \frac{d^2P(\Phi_t) }{dt^2}\le - \frac{\gamma^{-2}-D_2 ||\varphi||_\theta}{D_2 ||\varphi||_\theta^2}+ 
  C_0\gamma^2 ||\varphi||_\theta+ D_2 ||\psi ||_\theta+\gamma D_2 ||\varphi ||_\theta<0,
\]
if we do $||\varphi ||_\theta<\delta$, $||\psi ||_\theta<\delta$ and $\delta=\delta(C,\theta, \varepsilon, \gamma)>0$ is chosen sufficiently small.

\emph{Robustness of hypotheses}

\emph{Hyp. (i).} It is clear that hypothesis (i) is robust. 

\emph{Hyp. (\ref{hipu}).} We see that $\varphi\mapsto t_\varphi(\nu)$ is continuous, uniformly in $\nu$.
First, it is clear from (\ref{basic}) that 
\[
     | \log A_{-t \tilde{\varphi}}-\log A_{-t\varphi}|\le K || \tilde{\varphi}- \varphi ||,
\]
for some constant $K=K(\gamma)>0$ (and $t$ varying in a fixed bounded interval). Then, by definition of $t_\varphi(\nu)$, by (\ref{RR1}) and the above, we get
\begin{align*}
      &0=\int \log A_{-t_{\tilde{\varphi}}(\nu)\tilde{\varphi}}\,d\nu-\int \log A_{-t_\varphi(\nu)\varphi}\,d\nu\ge \\
       &\left| \int \log A_{-t_{\tilde{\varphi}}(\nu)\varphi} - \log A_{-t_\varphi(\nu)\varphi}\,d\nu \right|-
        \left| \int \log A_{-t_{\tilde{\varphi}}(\nu)\tilde{\varphi}} - \log A_{-t_{\tilde{\varphi}}(\nu)\varphi}\,d\nu \right| \\
        &\ge \gamma |t_{\tilde{\varphi}}(\nu)-t_{\varphi}(\nu) | - K(\gamma) || \tilde{\varphi}- \varphi ||,
\end{align*}
which proves the claimed.
Since the functions $\varphi\mapsto t_\varphi(\nu)$ and $\psi\mapsto \int \psi \,d\nu$ appearing in (\ref{pvorig2}) are continuous, uniformly in $\nu$, it follows that hypothesis (\ref{hipu}) is robust. 

\emph{Hyp. (ii).} The proof that $(\varphi, \psi)\mapsto \beta_{\varphi, \psi}(t)$ is continuous, uniformly for $t$ in a compact interval, is essentially contained in \cite{8}. In fact, let $t_0\in (\underline{t}_{\varphi}+\varepsilon,\overline{t}_{\varphi}-\varepsilon)$ and $\beta_0=\beta_{\varphi, \psi}(t_0)$. Then $F_{\varphi, \psi}(t_0,\beta_0)=0$ and, given $\eta>0$ sufficiently small, we have by (\ref{imp1}) and continuity (see Proposition 1 of \cite{8}) that there exists $\delta>0$ such that
\[
     F_{\tilde{\varphi}, \tilde{\psi}}(t,\beta_0-\eta)<0 \quad\text{and}\quad F_{\tilde{\varphi}, \tilde{\psi}}(t,\beta_0+\eta)>0
\]
for every $t\in (t_0-\delta, t_0+\delta)$, $||\tilde{\varphi}-\varphi||<\delta$ and $||\tilde{\psi}-\psi||<\delta$. So, by the intermediate value theorem, there is a unique $\tilde{\beta}_{\tilde{\varphi}, \tilde{\psi}}(t)\in (\beta_0-\eta,\beta_0+\eta)$
such that $F_{\tilde{\varphi}, \tilde{\psi}}(t,\tilde{\beta}_{\tilde{\varphi}, \tilde{\psi}}(t))=0$. By uniqueness, we have $\tilde{\beta}_{\tilde{\varphi}, \tilde{\psi}}(t)=\beta_{\tilde{\varphi}, \tilde{\psi}}(t)$ which implies the continuity of $(\varphi, \psi)\mapsto \beta_{\varphi, \psi}(t)$, uniformly for $t$ in a compact interval. Then it follows that hypothesis (ii) is robust.

\emph{Hyp. (iii).} Now we see that $(\varphi, \psi)\mapsto \beta'_{\varphi, \psi}(t)$ is continuous, uniformly for $t$ in a compact interval. This will follow by (\ref{imp3}) if we prove that the functions $(\varphi, \psi,t,\beta)\mapsto \frac{\partial F_{\varphi, \psi}}{\partial t}(t,\beta)$ and $\frac{\partial F_{\varphi, \psi}}{\partial \beta}(t,\beta)$ are continuous. From what has been said until now, it is clear that
\[
      \phi_{(\varphi, \psi,t,\beta)}=(t-D_{\varphi, \psi})\psi+\beta \log A_{-t\varphi}
\] 
is continuous. So the conclusion follows by (\ref{imp1}), (\ref{imp2}) and Proposition \ref{prop1}. Consequently, hypothesis (iii) is robust.

Therefore, if $\tilde\psi, \tilde\varphi$ are as described in statement of Theorem \ref{t5}, there is a unique measure of full dimension $\mu_{\tilde\psi,\tilde\varphi}$, and we can infer about its continuity.

\emph{Continuity of $\mu_{\psi,\varphi}$}.

Since 
\[
   (t,\varphi, \psi)\mapsto \Phi_{t,\varphi, \psi}=(t-D_{\varphi, \psi})\psi+\beta_{\varphi, \psi}(t) \log A_{-t\varphi}
\]
is continuous, we get that $(t,\varphi, \psi)\mapsto \nu_{\Phi_{t,\varphi, \psi}}$ is also continuous (see Proposition 1 of \cite{8}). By Proposition \ref{prop2} we have that
\[
         \mu_{\Phi_{t,\varphi, \psi}}=\mu_{t,\varphi, y}\times  \nu_{\Phi_{t,\varphi, \psi}},
\]
where $\{\mu_{t,\varphi, y}\}$ is the Gibbs family for $-t\varphi$. 
By Theorem 3.1 of \cite{4}, the Gibbs family  $\{\mu_{t,\varphi, y}\}$ is equal to the family of conditional measures $\{\mu_y\}$, on the fibers $\pi^{-1}(y)$, for the measure $\mu$ which is the classical Gibbs sate with respect to the H\"older-continuous potential $-t\varphi - P(\log A_{-t\varphi})$. Then it follows that $(t,\varphi)\mapsto \mu_{t,\varphi, y}$  
is continuous, uniformly in $y$, which implies the continuity of $(t,\varphi, \psi)\mapsto \mu_{\Phi_{t,\varphi, \psi}}$.

Finally, $\mu_{\varphi,\psi}$ is the measure $\mu_{\Phi_{t,\varphi, \psi}}$ where $t=t(\varphi,\psi)$ is the unique solution of 
$\frac{d}{dt} P(\Phi_{t,\varphi, \psi})=0$ (see Theorem \ref{tl2}). Since $\frac{d^2}{dt^2} P(\Phi_{t,\varphi, \psi})<0$ and
$(\varphi, \psi)\mapsto \frac{d}{dt} P(\Phi_{t,\varphi, \psi})$ is continuous, uniformly for $t$ in a compact interval, we get that $t(\varphi,\psi)$ is continuous, and so is $\mu_{\varphi,\psi}$.

\end{proof}

\begin{rem}\label{rem2}
It follows from the proof of Theorem \ref{t5} that the unicity of measure of maximal dimension and the robustness of hypotheses (i.e., everything except, possibly, the continuity of the measure) would also hold \emph{without the hypothesis of $-t\varphi$ being a basic potential}, if we could prove that, for some constant $C>0$,
\begin{align*}
  \int -\frac{d}{dt} \log A_{-t\varphi}\, d \nu_{\Phi_t} &\ge C^{-1} \min \varphi,\\
  \left| \int \frac{d^2}{dt^2} \log A_{-t\varphi}\, d \nu_{\Phi_t} \right| &\le C ||\varphi||_\theta.
\end{align*}
In this case, Theorem \ref{t2} (except, possibly, the continuity of the measure) would hold in the class of Non-linear Lalley-Gatzouras carpets.
\end{rem}

\section{Unique ergodic measure of full dimension}\label{s4}
In this section we prove Theorem \ref{t2}. 

Consider a non-trivial Sierpinski carpet. 
More precisely, consider the alphabet $\mathcal{I}=\{(i,j)\colon i\in\{1,...,m\}\text{ and } j\in\{1,...,m_i\}\}$ where $m$ and $m_i\ge2$ are natural numbers such that $m_i$ are not all equal to each other. For $(i,j)\in\mathcal{I}$, let
\[
         f^\circ_{ij}(x,y)=(a x + u_{ij}, b y + v_i),
\]
where $0<a<b<1$ and the positive numbers $v_i$ and $u_{ij}$ satisfy $b+v_i<v_{i+1}$, $a+u_{ij}<u_{ij+1}$ for all $(i,j)\in\mathcal{I}$, where $v_{m+1}=u_{im_i+1}=1$. Let $\Lambda^\circ$ be the corresponding attractor, i.e.
\[
    \Lambda^\circ=\bigcup_{(i,j)\in\mathcal{I}} f^\circ_{ij}(\Lambda^\circ).
\]

We will consider Non-linear Lalley-Gatzouras carpets $(f_{ij}, \Lambda)$, 
\[
      f_{ij}(x,y)=(a_{ij}(x,y), b_i(y)),\quad{(i,j)\in\mathcal{I}}
\]
which are close to $(f^\circ_{ij}, \Lambda^\circ)$.
Note that, since the alphabet $\mathcal{I}$ is fixed, all of these carpets are topologically modeled by the same Bernoulli shift $T\colon X\to X$, where $X=\mathcal{I}^\mathbb{N}$ and $T((i_1,j_1)(i_2,j_2)...)=((i_2,j_2)...)$, via the conjugacy
$h\colon X \to \Lambda$ given by
\[
       h((i_1,j_1)(i_2,j_2)...)=\bigcap_{n=1}^\infty f_{i_1j_1}\circ f_{i_2j_2}\circ\cdots \circ f_{i_nj_n} ([0,1]^2).
\]
Let $\pi\colon X\to Y$, where $Y=\{1,...,m\}^\mathbb{N}$ and $\pi((i_1,j_1)(i_2,j_2)...)=(i_1i_2...)$. Then $\pi\circ T=S\circ\pi$, where 
$S\colon \{1,...,m\}^\mathbb{N}\to\{1,...,m\}^\mathbb{N}$ is the Bernoulli shift given by  $S(i_1i_2...)=(i_2...)$.

By \cite{7} and \cite{8}, we have that
\[
   \hd\Lambda=\sup_{\mu\in\mathcal{M}(T)}\left\{\frac{h_{\mu\circ\pi^{-1}}(S)}{\int \psi\circ\pi\,d\mu}+\frac{h_{\mu}(T)-h_{\mu\circ\pi^{-1}}(S)}{\int \varphi\,d\mu}\right\},
\]
and that the measures of maximal dimension, as defined in previous section, being ergodic coincide with the \emph{ergodic measures of full dimension} (since the dimension of an ergodic measure is the expression between brackets in equation above). Here $\varphi\colon X\to \mathbb{R}$ and $\psi\colon Y\to\mathbb{R}$ are the positive and H\"older-continuous functions given by
\[
     \varphi((i_1,j_1)(i_2,j_2)...)=-\log \partial_x a_{i_1j_1} (x,y),\quad \psi(i_1i_2...)=-\log b'_i (y),
\]
where $(x,y)=h((i_1,j_1)(i_2,j_2)...)$. So we must see that we satisfy Theorem \ref{t5}'s hypotheses. 

We note that $-t\varphi$ is a basic potential (remember the definition from (\ref{basic})) if we restrict to the subclass of carpets $\mathcal{L}$, for then $\partial_x a_{i_1j_1} (x,y)$ does not depend on $x$. 

For the general Sierpinski carpet, we have that
\[
     t(\nu)=\frac{\sum_{i=1}^m p_i \log m_i}{\log a^{-1}},
\]
where $p_i$ is the $\nu$-measure of the cylinder $\{(i_1i_2...)\in Y \colon i_1=i\}$. Then, (\ref{pvorig2})
reads
\[
       D=\frac{1}{\log b^{-1}} \sup_{\nu\in\mathcal{M}(S)} \left\{h_{\nu}(S)+ \sum_{i=1}^m p_i \log m_i^\rho \right\},
\]
where $\rho=\frac{\log b}{\log a}$. It is well known that this supremum is attained at a Bernoulli measure $\nu$ with all $p_i>0$. Since the numbers $m_i$ are not all equal to each other, it follows that $\underline{t} <t(\nu)<\overline{t}$.
Hence we satisfy hypothesis (\ref{hipu}), for some $\varepsilon>0$.
Also note that $||\varphi^{\circ}||_\alpha=||\psi^{\circ}||_\alpha=0$.

Then, Theorem \ref{t2} follows from applying Theorem \ref{t5} to carpets in $\mathcal{L}$ which are $C^{1+\alpha}$ close to a non-trivial general Sierpinski carpet.

\end{document}